\documentclass{amsart}
\usepackage{graphicx}
\usepackage{latexsym}
\usepackage{hyperref}
\usepackage{amsfonts}
\usepackage{amssymb, mathrsfs, amsfonts, amsmath}
\usepackage{amsbsy}
\usepackage{amsfonts}
\usepackage{amssymb}
\newtheorem*{acknowledgement}{Acknowledgement}
\newtheorem{corollary}{Corollary}

\newtheorem{lemma}{Lemma}
\newtheorem{proposition}{Proposition}
\newtheorem{remark}{Remark}
\newtheorem{theorem}{Theorem}

\numberwithin{equation}{section}

\makeatletter
\@namedef{subjclassname@2020}{%
  \textup{2020} Mathematics Subject Classification}
\makeatother

\begin{document}
	
	\title[Quasi-Einstein manifolds]{Geometric inequalities for quasi-Einstein manifolds}

	\author{Rafael Di\'ogenes, Jaciane Gon\c calves and Ernani Ribeiro Jr}

	\address[R. Di\'ogenes]{UNILAB, Instituto de Ci\^encias Exatas e da Natureza, Rua Jos\'e Franco de Oliveira, s/n, 62790-970, Reden\c{c}\~ao / CE, Brazil.}\email{rafaeldiogenes@unilab.edu.br}
	
		\address[J. Gon\c calves]{Universidade Federal do Cear\'a - UFC, Departamento  de Matem\'atica, Campus do Pici, Av. Humberto Monte, Bloco 914, 60455-760, Fortaleza / CE, Brazil}\email{mjaciane05@alu.ufc.br}

	\address[E. Ribeiro Jr]{Universidade Federal do Cear\'a - UFC, Departamento  de Matem\'atica, Campus do Pici, Av. Humberto Monte, Bloco 914, 60455-760, Fortaleza / CE, Brazil}\email{ernani@mat.ufc.br}
	
		\thanks{R. Di\'ogenes was partially supported by CNPq/Brazil [Grant: 310680/2021-2].}
	
	\thanks{J. Gon\c calves was partially supported by CAPES/Brazil - Finance Code 001 and CNPq/Brazil.}	
	
	\thanks{E. Ribeiro was partially supported by CNPq/Brazil [309663/2021-0 \& 403344/2021-2], CAPES/Brazil and FUNCAP/Brazil [ITR-0214-00116.01.00/23].}
	
	\begin{abstract}
		In this article, we investigate certain geometric inequalities on quasi-Einstein manifolds. We use the generalized Reilly’s formulas by Qiu-Xia and Li-Xia to establish new boundary estimates and an isoperimetric type inequality for compact quasi-Einstein manifolds with boundary. Boundary estimates in terms of the first eigenvalue of the Jacobi operator and the Hawking mass are also established. In particular, we present a Heintze-Karcher type inequality for compact domains in quasi-Einstein manifolds. 
	\end{abstract}
	
	\date{\today}
	
	\keywords{quasi-Einstein manifolds; isoperimetric inequality; boundary estimates; Heintze-Karcher inequality.}
		
	\subjclass[2020]{Primary 53C20, 53C25; Secondary 53C65.}
	
	\maketitle
	
\section{Introduction}\label{SecInt}

A complete Riemannian manifold $(M^n,\,g),$ $n\geq 2,$ will be called \textit{m-quasi-Einstein manifold}, or simply \textit{quasi-Einstein manifold}, if there exists a smooth potential function $f$ on $M^n$ such that
\begin{equation}
\label{eq1F}
Ric_{f}^{m}=Ric+\nabla^2 f -\frac{1}{m}df \otimes df=\lambda g,
\end{equation} for some constants $\lambda$ and $m>0$ (cf. \cite{Besse,Case-Shu-Wei}). Here, $\nabla^2 f$ stands for the Hessian of $f$ and $Ric$ is the Ricci tensor of $g.$ In this case, $f$ is a quasi-Einstein potential. We say that a quasi-Einstein manifold is \textit{trivial} if its potential function $f$ is constant, otherwise, we say that it is \textit{nontrivial}. 

The $m$-Bakry-\'Emery Ricci tensor $Ric_{f}^{m}$ is a natural object in geometric analysis and it is directly related to the study of smooth metric measure spaces and diffusion operators by Bakry and \'Emery \cite{BaEm}. Observe that $\infty$-quasi-Einstein manifold is precisely a gradient Ricci soliton. Moreover, when $m=1,$ we consider in addition that  $\Delta e^{-f}+\lambda e^{-f}=0$ in order to obtain the {\it static spaces}, which have special interest due to the connection with general relativity. 
%This terminology relies on the physical nature of the problem associated to the vacuum Einstein field equations. 
As discussed by Besse \cite[pg. 267]{Besse}, a quasi-Einstein manifold is a base of a warped product Einstein metric. Besides, quasi-Einstein manifolds have attracted interest in physics due to their relation with the geometry of a degenerate Killing horizon and horizon limit (see, e.g., \cite{BGKW,BGKW2,Wylie}). Notice that, for $m<\infty,$ we may consider the function $u=e^{-\frac{f}{m}}$ on $M^n,$ which in turn implies that $\nabla u=-\frac{u}{m}\nabla f$ and hence, (\ref{eq1F}) becomes
\begin{equation}
\label{eq1F2}
\nabla^2u=\frac{u}{m}(Ric-\lambda g).
\end{equation}

In \cite{He-Petersen-Wylie,He-Petersen-Wylie-2}, He, Petersen and Wylie started a fruitful study of quasi-Einstein manifolds with boundary. In this context, a Riemannian manifold $(M^n,\,g)$ with (nonempty) boundary $\partial M$ is an \textit{m-quasi-Einstein manifold} if there exists a smooth potential function $u$ on $M^n$ satisfying the system
\begin{equation}
\label{eq-quasi-einstein}
\left\{%
\begin{array}{lll}
    \displaystyle \nabla^{2}u = \dfrac{u}{m}(Ric-\lambda g) & \hbox{in $M,$} \\
    \displaystyle u>0 & \hbox{on $int(M),$} \\
        \displaystyle u=0 & \hbox{on $\partial M.$} \\
    \end{array}%
\right.
\end{equation} According to \cite[Theorem 4.1]{He-Petersen-Wylie}, nontrivial compact $m$-quasi-Einstein manifolds with (non-empty) boundary $\partial M$ have necessarily $\lambda>0.$ Nontrivial examples of compact and noncompact quasi-Einstein manifolds $(M^n,\,g,\,u,\,\lambda)$ can be found in, e.g., \cite{Besse,Bohm,Bohm2,Case,Case2,Case-Shu-Wei,Catino-Carlo-Mazziere-Rimoldi,RTE,He-Petersen-Wylie,Lu,Ernani_Keti,Rimoldi,L-F-Wang}. For what follows, it is important to recall some of them. 
\begin{enumerate}
\item {\it Compact with boundary}:
\begin{itemize}
\item[(i)] the standard hemisphere $\left(\mathbb{S}_+^n, \, dr^2+\sin^2(r) g_{\mathbb{S}^{n-1}},\, u=\cos(r),\,\lambda>0 \right),$ where $r$ is a height function with $r\leq \frac{\pi}{2};$ 
\item[(ii)]  $\left(I\times\Bbb{S}^{n-1}, \,dt^2+\frac{n-2}{\lambda}g_{\Bbb{S}^{n-1}},\, u(t,x)=\sin\left(c\,t\right),\,\lambda>0\right),$ where $c>0$ is a cons\-tant; 
\item[(iii)]    $\left(\Bbb{S}^{p+1}_+\times\Bbb{S}^q,\,dr^2+\sin^2r g_{\Bbb{S}^p}+\frac{q-1}{p+m}g_{\Bbb{S}^q},\,u(x,y)=\cos (r(x)), \,\lambda=p+m\right),$
where $r$ is a height function on $\Bbb{S}^{p+1}_+$ with $r\leq \frac{\pi}{2}$ and $q>1.$
\end{itemize}

\item {\it Noncompact}:
\begin{itemize}
\item[(a)]  $\left([0,\infty)\times F,\, dt^2+g_{F},\, u(t,x)=Ct,\,\lambda=0\right);$ 
\item[(b)]  the hyperbolic space $\left(\mathbb{H}^n, \ dt^2+\sqrt{-k}\sinh^2(\sqrt{-k} t)g_{\mathbb{S}^{n-1}},\,u=C\cosh (\sqrt{-k}t),\, \lambda<0\right);$ 
\item[(c)]  $\left([0,\infty)\times N, dt^2+\sqrt{-k}\cosh^2(\sqrt{-k}t)g_{\mathbb{S}^{n-1}},\,u(t,x)=C\sinh(\sqrt{-k}t), \,\lambda<0\right);$ 
\item[(d)]  $\left(\mathbb{R}\times F, \ dt^2+e^{2\sqrt{-k}t}g_{F},\, u(t,x)=Ce^{2\sqrt{-k}t},\, \lambda<0\right),$ 
\end{itemize} where $F$ is Ricci flat, $N$ is an Einstein metric with negative Ricci curvature, $C$ is an arbitrary positive constant and $k=\frac{\lambda}{m+n-1}.$ 
\end{enumerate} 

He, Petersen and Wylie \cite{He-Petersen-Wylie} proved that a nontrivial quasi-Einstein manifold with cons\-tant Ricci curvature must be isometric to either Example
${\rm (i)},$ or ${\rm (a)},$ or ${\rm (b)},$ or ${\rm (c)},$ or ${\rm (d)}.$ In the case of compact manifolds with boundary, Costa, Ribeiro and Zhou \cite{CRZ2024} recently showed that a $3$-dimensional simply connected compact quasi-Einstein manifold with boundary and constant scalar curvature must be isometric to either Example ${\rm (i)},$ or ${\rm (ii)}.$ For dimension $n=4,$ they proved that a $4$-dimensional simply connected compact quasi-Einstein manifold with boundary and constant scalar curvature is isometric to either Example ${\rm (i)},$ or ${\rm (ii)},$ or  ${\rm (iii)}.$ Despite of these results, the classification of compact quasi-Einstein manifolds with boundary and constant scalar curvature remains open for dimensions $n\geq 5.$

Geometric inequalities are fundamental objects of study in geometry. Such inequalities are useful in proving novel obstruction results and put away possible new examples of a special kind of manifolds (or metric). Among the geometric inequalities that motivate the present work, we highlight the isoperimetric and Heintze-Karcher inequalities. The isoperimetric inequality is one of the oldest and most famous inequality in geometry. In $\mathbb{R}^{n},$ the isoperimetric inequality asserts that if $M\subset\mathbb{R}^{n}$ is a compact domain with smooth boundary $\partial M,$ then
\begin{equation}
|\partial M|\geq \frac{|\partial\mathbb{B}_{1}^{n}|}{Vol(\mathbb{B}_{1}^{n})^{\frac{n-1}{n}}}\, Vol(M)^{\frac{n-1}{n}},
\end{equation} where $|\partial M|$ denotes the $(n-1)$-dimensional volume of $\partial M$ and $Vol(M)$ is the volume of $M.$ In special cases, the $n$-dimensional isoperimetric inequality is equivalent to the Sobolev inequality on $\Bbb{R}^n$ (see \cite{Ma} and \cite{Ros}). On the other hand, the Heintze-Karcher inequality asserts that, for a bounded domain $\Omega$ in $\Bbb{R}^n$ with smooth and strictly mean-convex boundary $\partial \Omega,$ namely, $H>0$ on $\partial \Omega,$ 	
\begin{equation}
\label{eqHK1}
\int_{\partial \Omega}\frac{1}{H}\,dS \geq \frac{n+1}{n}|\Omega|,
 \end{equation} where $H$ is the mean curvature of the $\partial \Omega.$ This inequality was first outlined by Heintze and Karcher \cite{HK} in 1978 and in the form (\ref{eqHK1}) by Ros \cite{Ros2} in 1987.

 In recent years, significant progress has been made in the study of boundary and volume estimates for special classes of manifolds (or metrics), such as static spaces, $V$-static spaces, Einstein-type manifolds and critical metrics. For instance, motivated by the classical isoperimetric inequality and a result due Shen \cite{Shen} and Boucher, Gibbons and Horowitz \cite{BGH}, it was  established some boundary and volume estimates for critical metrics of the volume funcional in \cite{Baltazar-Diogenes-Ribeiro,BLF20,BS,BDRR,Corvino-Eichmair-Miao,FY,Miao-Tam 2009,Yuan} and for static spaces in \cite{Lucas,BMa,BFP,BGH,CDPR,HMR,Kwong-Miao,Qiu-Xia}, etc. In the same spirit, Di\'ogenes, Gadelha and Ribeiro \cite{RTE2} showed that a compact $m$-quasi-Einstein manifold with connected boundary and constant scalar curvature must satisfy
		\begin{eqnarray}\label{eqT1}
			|\partial M|\leq n\sqrt{\frac{\lambda}{m+n-1}}Vol(M).
		\end{eqnarray} 
		
		In our first result, we shall establish a new estimate for the area of the boundary of a compact $m$-quasi-Einstein manifold. In particular, no assumption on the scalar curvature is assumed. More precisely, we have the following result.

	\begin{theorem}
	\label{Theo2}
		 Let $(M^n,\,g,\,u,\,\lambda),$ $n\geq 3,$ be a compact, oriented $m$-quasi-Einstein manifold with connected boundary and $m>1.$ Then we have
		 \begin{eqnarray}
		 \label{eqTeo2}
		 	|\partial M|\leq \frac{n\lambda}{\sqrt{\left(\lambda_1(m+n-1)+\lambda\right)(m+n-1)}}Vol(M),
		 \end{eqnarray} where $\lambda_1$ is the first non-zero eigenvalue of the Laplacian operator. Moreover, if equality holds in (\ref{eqTeo2}), then $(M^n,\,g)$ is isometric, up to scaling, to the standard hemisphere $\mathbb{S}^n_+.$
	\end{theorem}
	
\begin{remark}
Notice that the estimate (\ref{eqTeo2}) improves (\ref{eqT1}). Indeed, we have
\begin{equation*}
\frac{n\lambda}{\sqrt{(\lambda_1(m+n-1)+\lambda)(m+n-1)}}< \frac{n\lambda}{\sqrt{\lambda(m+n-1)}}= n\sqrt{\frac{\lambda}{m+n-1}}.
\end{equation*}
\end{remark}

In the sequel, we investigate an isoperimetric type inequality for compact $m$-quasi-Einstein manifolds with connected boundary and constant scalar curvature. To do so,  we first establish a lower bound estimate (depending on the volume of $M^n$) for the area of the boundary. To be precise, we have the following result.

\begin{theorem}
\label{theo4}
	 Let $(M^n,\,g,\,u,\,\lambda),$ $n\geq 3,$ be a compact, oriented $m$-quasi-Einstein manifold with connected boundary, constant scalar curvature $R$ and $m>1.$ Then we have:
	 	\begin{equation}\label{eqtheo4}
	 	|\partial M|\geq \frac{(m-1)\alpha}{(m-2)\alpha+\lambda m}\sqrt{\frac{2\lambda mn+\alpha\big(n(m-1)+4\big)}{2n(m+1)m}}Vol(M),
	 \end{equation}
	 where  $\alpha=n\lambda-R$ is a positive constant. Moreover, if equality holds in (\ref{eqtheo4}), then $M^n$ is isometric, up to scaling, to the standard hemisphere $\mathbb{S}^n_+.$
\end{theorem}

A relevant observation is that the proof of Theorem \ref{theo4} is inspired by the work of Baltazar, Di\'ogenes and Ribeiro \cite{Baltazar-Diogenes-Ribeiro} combined with a generalized Reilly's formula obtained by Qiu and Xia \cite{Qiu-Xia}. As a consequence of Theorem \ref{theo4}, we obtain the following isoperimetric type inequality for compact $m$-quasi-Einstein manifolds with connected boundary and constant scalar curvature.

\begin{corollary}
\label{corol1}
	Let $(M^n,\,g,\,u,\,\lambda),$ $n\geq 3,$ be a compact, oriented $m$-quasi-Einstein manifold with connected boundary, constant scalar curvature $R$ and $m>1.$ Then we have:
	\begin{equation}
	\label{cor1}
		|\partial M|\geq \Lambda(\alpha,m,n,u)^{\frac{1}{n}}Vol(M)^{\frac{n-1}{n}},
	\end{equation} where 
\begin{equation*}
\label{}
	\Lambda(\alpha,m,n,u) =\frac{\alpha}{m|\nabla u|_{\mid_{_{\partial M}}}}\left(\frac{(m-1)\alpha}{(m-2)\alpha+\lambda m}\sqrt{\frac{2\lambda mn+\alpha(n(m-1)+4)}{2n(m+1)m}}\right)^{n-1}\int_{M}udV_g
\end{equation*} and $\alpha=n\lambda-R.$ Moreover, if equality holds in (\ref{cor1}), then $M^n$ is isometric, up to scaling, to the standard hemisphere $\mathbb{S}^n_+.$
 \end{corollary}

	\begin{remark}
	\label{rem2a}
	It would be interesting to investigate whether the constant $\Lambda(\alpha,m,n,u)$ in (\ref{cor1}) can be improved to depend only on the dimension $n,$ $m$ and the volume of the unit ball. We highlight that $|\nabla u|_{\mid_{_{\partial M}}}$ is a positive constant along $\partial M$ (see Propositions 2.2 and 2.3 in \cite{He-Petersen-Wylie}). Also, it follows from \cite[Corollary 4.3]{He-Petersen-Wylie} that if a nontrivial compact quasi-Einstein manifold with (nonempty) boundary has constant scalar curvature $R,$ then $R < n\lambda.$ Whence, $\alpha>0.$
	\end{remark}
		
It is known that the boundary $\partial M$ of a quasi-Einstein manifold is totally geodesic with the induced metric (see Section \ref{sec2}). Therefore, it is natural to seek for an estimate for the area of the boundary $\partial M$ in terms of the eigenvalue of Jacobi operator of $\partial M.$ We recall that given an arbitrary function $\varphi\in C^{\infty}(\partial M),$ the Jacobi operator (or stability operator) $J$ acting in $\varphi$ is given by

\begin{eqnarray*}
    J({\varphi})=\Delta_{_{\partial M}}\varphi +(Ric(\nu,\nu)+|\mathbb{II}|^2)\varphi,
\end{eqnarray*} where $\Delta_{_{\partial M}}$ stands for the Laplacian operator on $\partial M$, $Ric(\nu,\nu)$ is the Ricci curvature of $M$ in the direction of the outward unit normal vector field $\nu$ and $\mathbb{II}$ is the second fundamental form of $\partial M.$ Besides, let $\beta_1$ be the first eigenvalue of the Jacobi operator $J,$ i.e., $J(\varphi)=-\beta_1 \varphi.$ Thus, $\beta_1$ is given by
\begin{eqnarray}\label{eq-egvle}
    \beta_1=\inf_{\varphi\neq 0} \frac{-\int_{\partial M}\varphi J(\varphi) dS_g}{\int_{\partial M}\varphi^2 dS_g}.
\end{eqnarray} Specifically, the \textit{index} of $\partial M$ is the number of (counted with multiplicity) negative eigenvalues of $J.$ In \cite[Theorem 1.10]{BS}, Barros and Silva proved an estimate for the area of the boundary of a static space involving the first eigenvalue of the Jacobi operator, provided that the boundary $\partial M$ is Einstein and $\inf R^{\partial M}>0.$ A similar result was obtained by Costa et al. \cite{CDPR} for a static perfect fluid space-time under the same conditions. In the following, we shall establish a boundary estimate for a compact quasi-Einstein manifold with boundary in terms of the first eigenvalue of the Jacobi operator $\beta_1.$ 

\begin{theorem}
\label{thmJac}
Let $(M^n,\,g,\,u,\,\lambda),$ $n\geq 3,$ be a compact $m$-quasi-Einstein manifold with connected Einstein boundary and $m>1.$ Then we have:
\begin{equation}
\label{eqk8Er}
|\partial M|\leq \left(\frac{(n-1)(n-2)(m+n-1)}{2(m+n-1)\beta_{1}+n(n-1)\lambda}\right)^{\frac{n-1}{2}} \omega_{n-1},
\end{equation} where $\omega_{n-1}$ denotes the volume of the round unit sphere $\Bbb{S}^{n-1}.$ Moreover, if equality holds in (\ref{eqk8Er}), then $\partial M$ is isometric to the round sphere $\Bbb{S}^{n-1}.$

\end{theorem}

We remark that the right hand side of (\ref{eqk8Er}) is positive due to a lower bound of $\beta_{1}$ proved in Proposition \ref{proklk}. One new feature in Theorem \ref{thmJac} is that no scalar curvature condition is assumed. Furthermore, as we will see in Remark \ref{remZ}, the estimate obtained in Theorem \ref{thmJac} improves the estimate established in \cite[Theorem 1]{RTE2}.

In \cite{RTE2}, Di\'ogenes, Gadelha and Ribeiro proved a boundary estimate for compact quasi-Einstein manifolds with boundary in terms of the Brown-York mass. Here, as a consequence of the arguments in the proof of Theorem \ref{thmJac}, we shall obtain boundary estimates for $3$-dimensional compact  quasi-Einstein manifolds involving the Hawking mass. The Hawking mass\footnote{The Hawking mass is often utilized as a lower bound for the Bartnik quasi-local mass (see \cite{MiaoM,Mondino}).} of a $2$-surface $\Sigma$ is given by 
\begin{equation}
\mathfrak{m}_H(\Sigma)=\frac{|\Sigma|^{\frac{1}{2}}}{(16\pi)^{\frac{3}{2}}}\left(8\pi \chi(\Sigma)-\int_{\Sigma}H^2 dS_{g}-\frac{2}{3}\Lambda |\Sigma|\right),
\end{equation} where $\Lambda=\inf_{M}R$ and $\chi(\Sigma)$ is the Euler characteristic of $\Sigma;$ see \cite{BBT,Hawking}. The problem of finding explicit lower bounds on Hawking mass is very intriguing; see, e.g., \cite{Benedito,MiaoM,Mondino}. More precisely, we have the following result.

\begin{corollary}
\label{corJac}
Let $(M^3,\,g,\,u,\,\lambda)$ be a $3$-dimensional compact $m$-quasi-Einstein manifold with connected boundary and $m>1.$ Then the following assertions hold:
\begin{enumerate}
\item[(i)] 
\begin{equation}
\label{eqkln568}
\mathfrak{m}_H(\partial M)+\frac{4\lambda}{(m+2)}\left(\frac{|\partial M|}{16\pi}\right)^{\frac{3}{2}}\leq \frac{1}{2}\sqrt{\frac{(m+2)}{\lambda}}.
\end{equation} Moreover, if equality holds in (\ref{eqkln568}), then $\partial M$ is isometric to the round sphere $\mathbb{S}^2.$

\item[(ii)] 
\begin{equation}
\label{eqkln568PP}
\mathfrak{m}_H(\partial M)\geq 4\left(\beta_{1}+\frac{2\lambda}{m+2}\right) \left(\frac{|\partial M|}{16\pi}\right)^{\frac{3}{2}} \geq 0.
\end{equation}
\end{enumerate}
\end{corollary}
	
	\vspace{0.20cm}
	
	\begin{remark}
Note that (\ref{eqkln568PP}) can be seen as a Penrose inequality for the boundary (see \cite{GI}).
	\end{remark}

We now turn our attention to the Heintze-Karcher inequality. Recently, there has been significant progress concerning the Heintze-Karcher type inequalities. For instance, Brendle \cite{Brendle} obtained Heintze-Karcher type inequalities on substatic warped product spaces. By using an elliptic method, Li and Xia \cite{LX2,LX} established Heintze-Karcher type inequalities for bounded domains in general substatic Riemannian manifolds. Among other results, they proved that letting $\Omega\subset M$ be a bounded domain with connected strictly mean-convex boundary $\partial \Omega$ such that $f|_{_{\Omega}}>0$ on an $n$-dimensional substatic Riemannian manifold $(M^n,\,g,\,f),$ then the following Heintze-Karcher type inequality holds:
\begin{equation}
\label{eq1HKLX}
\int_{\partial \Omega}\frac{f}{H}\,dS\geq \frac{n}{n-1}\int_{\Omega}f\,d\Omega.
\end{equation} Moreover, if equality holds in (\ref{eq1HKLX}), then $\partial \Omega$ is umbilical (see \cite[Theorem 1.3]{LX}). Furthermore, a rigidity statement for the equality case of (\ref{eq1HKLX}) was established by Borghini, Fogagnolo and Pinamonti in \cite{BFP}; see also related results in \cite{JWXZ}.

In our next result, as an application of the generalized Reilly’s formula obtained by Li and Xia \cite{LX}, we shall establish an Heintze-Karcher type inequality for a compact domain on a Riemannian manifold satisfying (\ref{eq1F2}).

\begin{theorem}
	\label{teoHK1}
	Let $(M^n,\,g,\,u,\,\lambda)$, $n\geq 3,$ be an $n$-dimensional (compact or noncompact) $m$-quasi-Einstein manifold with $m>1$ and $Ric \geq \frac{(n-1)}{m+n-1}\lambda.$ Consider $\Omega\subset int(M)$ be a compact domain with connected strictly mean-convex boundary $\partial \Omega.$ Then we have 
	\begin{equation}
	\label{eqteo2}
	\int_{\partial \Omega}\frac{u}{H}\,dS\geq \frac{n}{n-1}\int_{\Omega}u\,d\Omega.
		\end{equation}
	Moreover, if equality holds in (\ref{eqteo2}), then $\Omega$ is a geodesic ball and $(M^n,\,g)$ is isometric, up to scaling, to either Example
${\rm (i)},$ or ${\rm (a)},$ or ${\rm (b)},$ or ${\rm (c)},$ or ${\rm (d)}.$ 
\end{theorem}

\begin{remark}
As we shall see later,  a quasi-Einstein manifold satisfying the Ricci bound condition assumed in Theorem \ref{teoHK1} must be a substatic Riemannian manifold. Consequently, the Heintze-Karcher type inequality may follow from Li and Xia \cite{LX}. An  advantage here is that the rigidity statement for the equality case of (\ref{eqteo2}) is established. We also observe that such a Ricci bound condition implies that $R\geq \frac{n(n-1)}{m+n-1}\lambda,$ which holds in general for compact quasi-Einstein manifolds; see \cite[Remark 5.1]{He-Petersen-Wylie} and \cite[Proposition 3.6]{Case-Shu-Wei}. Anyway, it should be interesting to obtain an Heintze-Karcher type inequality for compact domain on quasi-Einstein manifolds by removing the Ricci bound condition. 
\end{remark}

The rest of this paper is organized as follows. In Section \ref{sec2}, we review some basic facts on $m$-quasi-Einstein manifolds. Moreover, we present a couple of key lemmas that will be used in the proofs of the main results. Section \ref{sec3} collects the proofs of Theorems \ref{Theo2}, \ref{theo4}, \ref{thmJac} and \ref{teoHK1} and Corollaries \ref{corol1} and \ref{corJac}.

\section{Background}
\label{sec2}

In this section, we review some basic facts and a couple of lemmas that will be useful in the proof of the main results. We start by recalling that the fundamental equation of an $m$-quasi-Einstein manifold $(M^n,\,g,\,u,\,\lambda)$ (possibly) with boundary $\partial M$ is given by 
\begin{equation}\label{fund-equation}
\nabla^2u=\frac{u}{m}(Ric-\lambda g),
\end{equation}
where $u>0$ in the interior of $M^n$ and $u=0$ on $\partial M.$ In particular, taking the trace of (\ref{fund-equation}) we arrive at 
\begin{equation}\label{eq-laplac}
	\Delta u=\frac{u}{m}(R-\lambda n).
\end{equation} Plugging this fact into (\ref{fund-equation}) yields
\begin{equation}\label{RicciSemtrac}
	u\mathring{Ric}=m\mathring{\nabla^2}u,
\end{equation} where $\mathring{T}=T-\frac{trT}{n}g$ stands for the traceless part of $T$.

Since $u>0$ in the interior of $M^n$ and $u=0$ on the boundary $\partial M$, one deduces that $\nu=-\frac{\nabla u}{|\nabla u|}$ is the outward unit normal vector. Besides, it follows from \cite[Propositions 2.2 and 2.3]{He-Petersen-Wylie} that $|\nabla u|\neq 0$ is constant along $\partial M.$ Thereby, we set an orthonormal frame given by $\left\{e_1,\ldots,e_{n-1},e_n=\nu\right\}.$ The second fundamental form at $\partial M$ satisfies
\begin{eqnarray}\label{fundamental second}
\mathbb{II}_{ij}=\langle \nabla _{e_{i}}\nu, e_{j}\rangle= -\frac{1}{|\nabla u|}\nabla _{i}\nabla _{j} u=0,
\end{eqnarray} for any $1\leq i,j\leq n-1.$ Consequently, $\partial M$ is totally geodesic. By Gauss equation $$R^{\partial M}_{ijkl}=R_{ijkl}-h_{il}h_{jk}+h_{ik}h_{jl},$$ we then infer
\begin{eqnarray}
 \label{GaussEq}
R^{\partial M}_{ijkl}&=&R_{ijkl}.
\end{eqnarray} Taking the trace in (\ref{GaussEq}), we have

\begin{eqnarray}
\label{RicciBordoM}
R^{\partial M}_{ik}&=&R_{ik}-R_{inkn}
\end{eqnarray}
and
\begin{eqnarray}
\label{EscalarBordoM}
R^{\partial M}&=&R-2R_{nn}.
\end{eqnarray}

We now collect some well-known equations for quasi-Einstein mani\-folds (cf. \cite{Case-Shu-Wei,DG 2019,KimKim}).

\begin{proposition}
\label{propA}
	Let $(M^n,\,g,\,u,\,\lambda),$ $n\geq 3,$ be an $n$-dimensional quasi-Einstein manifold. Then we have:
\begin{equation}
\label{eqK1a}
	\frac{1}{2}u\nabla R=-(m-1)Ric(\nabla u)-(R-(n-1)\lambda)\nabla u;
\end{equation}

\begin{equation}
\label{eqmu}
	u\Delta u+(m-1)|\nabla u|^2+\lambda u^2 =\mu,
\end{equation} where $\mu$ is a constant;
	
	\begin{eqnarray}
	\label{laplacianoR}
		\frac{1}{2}\Delta R+\frac{m+2}{2u}\langle \nabla u,\nabla R\rangle&=& -\frac{m-1}{m}\left|Ric-\frac{R}{n}g\right|^2\\
		&&-\frac{(n+m-1)}{nm}(R-n\lambda)\left(R-\frac{n(n-1)}{n+m-1}\lambda\right).\nonumber
	\end{eqnarray}
\end{proposition}

As a consequence of Proposition \ref{propA}, we obtain the following lemma. 

\begin{lemma}
\label{lem1A1}
Let $(M^n,\,g,\,u,\,\lambda),$ $n\geq 3,$ be an $n$-dimensional quasi-Einstein manifold with boundary $\partial M$ and $m>1.$ Then, at $\partial M,$ we have:
\begin{equation*}
(m+1)R=(m-1)R^{\partial M} + 2(n-1)\lambda.
\end{equation*}
\end{lemma}
\begin{proof}
Since $u$ vanishes on the boundary $\partial M,$ we may use (\ref{eqK1a}) to infer 
\begin{eqnarray*}
Ric(\nabla u,\nabla u)=-\frac{\left(R-(n-1)\lambda\right)}{m-1}|\nabla u|^2
\end{eqnarray*} on $\partial M.$ Consequently,
\begin{eqnarray*}
R_{nn}=-\frac{\left(R-(n-1)\lambda\right)}{m-1}.
\end{eqnarray*} Hence, by using (\ref{EscalarBordoM}), 
\begin{equation*}
\frac{R-R^{\partial M}}{2}=-\frac{\left(R-(n-1)\lambda\right)}{m-1}
\end{equation*} and the result follows. 
\end{proof}

It is known from \cite[Remark 5.1]{He-Petersen-Wylie} that the scalar curvature $R$ of a nontrivial compact quasi-Einstein manifold $M^n$ with boundary $\partial M$ must satisfy 
\begin{align}
\label{limitacaoR}
	R\geq \frac{n(n-1)}{m+n-1}\lambda.
\end{align} In particular, it follows from Eq. (\ref{laplacianoR}) that if equality occurs in (\ref{limitacaoR}), then $M^n$ is necessarily Einstein. As observed in \cite[Proposition 3.6]{Case-Shu-Wei}, (\ref{limitacaoR}) also holds in the case of compact without boundary.

As mentioned in the Introduction, the Jacobi operator $J$ acting in $\varphi$ is given by

\begin{eqnarray*}
    J({\varphi})=\Delta_{_{\partial M}}\varphi +(R_{nn}+|\mathbb{II}|^2)\varphi,
\end{eqnarray*} for all $\varphi\in C^{\infty}(\partial M).$ Besides, we consider $\beta_1$ to be the first eigenvalue of the Jacobi operator $J,$ i.e.,
\begin{eqnarray}\label{eq-egvle11}
    \beta_1=\inf_{\varphi\neq 0} \frac{-\int_{\partial M}\varphi J(\varphi) dS_g}{\int_{\partial M}\varphi^2 dS_g}.
\end{eqnarray} With aid of this notation, we get the following proposition.

\begin{proposition}
\label{proklk}
Let $(M^n,\,g,\,u,\,\lambda),$ $n\geq 3,$ be a compact $n$-dimensional quasi-Einstein manifold with boundary $\partial M$ and $m>1.$ Then we have:

\begin{equation*}
\beta_{1}\geq -\frac{(n-1)}{m+n-1}\lambda.
\end{equation*} Moreover, equality holds if and only if $R=\frac{n(n-1)}{m+n-1}\lambda$ on $\partial M$ and the eigenfunction associated to $\beta_{1}$ is constant. 
\end{proposition} 

\begin{proof}
Firstly, on integrating by parts, we have

\begin{eqnarray}
\label{eqhjmnbv098}
-\int_{\partial M}\varphi J(\varphi) dS_{g}&=&-\int_{\partial M}\varphi\Delta_{_{\partial M}}\varphi\, dS_{g}-\int_{M}R_{nn}\varphi^2 \,dS_{g}\nonumber\\&=& \int_{\partial M}|\nabla_{_{\partial M}}\varphi|^2 \,dS_{g}-\int_{M}R_{nn}\varphi^2 \,dS_{g}.
\end{eqnarray}

On the other hand, by using (\ref{eqK1a}) on $\partial M$ and (\ref{limitacaoR}), one obtains that

\begin{eqnarray}
\label{eqhjmnbv098AA}
R_{nn}&=& \frac{(n-1)\lambda -R}{m-1}\nonumber\\&\leq & \frac{1}{m-1}\left((n-1)\lambda -\frac{n(n-1)}{m+n-1}\lambda\right)\nonumber\\&=&\frac{(n-1)}{m+n-1}\lambda.
\end{eqnarray} Plugging this into (\ref{eqhjmnbv098}) yields

\begin{eqnarray*}
-\int_{\partial M}\varphi J(\varphi) dS_{g}\geq \int_{\partial M}|\nabla_{_{\partial M}}\varphi|^2 \,dS_{g} -\frac{(n-1)}{m+n-1}\lambda\int_{\partial M}\varphi^2\, dS_{g}
\end{eqnarray*} for all $\varphi\in C^{\infty}(\partial M).$ In particular, choosing $\varphi$ such that $J(\varphi)=-\beta_{1}\varphi,$ one sees that 

\begin{eqnarray}
\label{eqhjmnbv098BB}
\beta_{1}&\geq & \frac{\int_{\partial M}|\nabla_{_{\partial M}}\varphi|^2 \,dS_{g}}{\int_{\partial M}\varphi^2\, dS_{g}}-\frac{(n-1)}{m+n-1}\lambda\nonumber\\&\geq & -\frac{(n-1)}{m+n-1}\lambda.
\end{eqnarray} The case of equality follows directly from (\ref{eqhjmnbv098AA}) and (\ref{eqhjmnbv098BB}). So, the proof is completed. 
\end{proof}

\begin{remark}
\label{remZ}
It is easy to check from Proposition \ref{proklk} that

\begin{equation}
\frac{2}{n-1}\beta_{1}+\frac{n}{m+n-1}\lambda\geq \frac{(n-2)}{m+n-1}\lambda.
\end{equation} This implies that our estimate obtained in Theorem \ref{thmJac} improves the estimate established in \cite[Theorem 1]{RTE2}.
\end{remark}

By assuming that the scalar curvature of the boundary $R^{\partial M}$ of a quasi-Einstein manifold is constant, we also obtain the following result.

\begin{proposition}
Let $(M^n,\,g,\,u,\,\lambda),$ $n\geq 3,$ be an $n$-dimensional quasi-Einstein manifold with boundary $\partial M$ and $m>1.$ Suppose that the scalar curvature of $\partial M$ is constant. Then

$$\beta_{1}=\frac{R-(n-1)\lambda}{m-1}.$$

\end{proposition}
\begin{proof}
Choosing $\varphi\equiv 1$ in (\ref{eq-egvle11}), we obtain 

\begin{equation*}
\beta_{1}|\partial M|\leq -\int_{\partial M}R_{nn}\,dS_{g} 
\end{equation*} and by using (\ref{eqK1a}) on $\partial M,$ one deduces that
 \begin{equation*}
\beta_{1}|\partial M|\leq \int_{\partial M}\left(\frac{R-(n-1)\lambda}{m-1}\right) \,dS_{g}.
\end{equation*} By Lemma \ref{lem1A1}, we have that $R$ is constant on $\partial M$ and hence,

\begin{equation*}
\beta_{1}\leq \frac{R-(n-1)\lambda}{m-1}.
\end{equation*}

Next, we derive the reverse inequality. It follows from  (\ref{eqhjmnbv098}) and (\ref{eqK1a}) that 

\begin{equation*}
-\int_{\partial M}\varphi J(\varphi) dS_{g}=\int_{\partial M}|\nabla_{_{\partial M}}\varphi|^2 \,dS_{g} +\frac{R-(n-1)\lambda}{m-1}\int_{\partial M}\varphi^2\,dS_{g}.
\end{equation*} Now, taking $\varphi$ such that $J(\varphi)=-\beta_{1}\varphi,$ we infer

\begin{eqnarray}
\beta_{1}&=&\frac{\int_{\partial M}|\nabla_{_{\partial M}}\varphi|^2 \,dS_{g}}{\int_{\partial M}\varphi^2\, dS_{g}} + \frac{R-(n-1)\lambda}{m-1}\nonumber\\&\geq & \frac{R-(n-1)\lambda}{m-1}\nonumber.
\end{eqnarray} Thus, we obtain the asserted equality. 

\end{proof}

In order to proceed, we recall a very important result so called {\it generalized Reilly's formula} that was established by Qiu and Xia \cite{Qiu-Xia}. 

\begin{proposition}[\cite{Qiu-Xia}]
\label{prop-qui-xia}
Let $(M^n,\,g)$ be an $n$-dimensional compact Riemannian manifold with boundary $\partial M.$ Given two functions $f$ and $u$ on $M$ and $k\in \Bbb{R},$ we have the following identity: 

\begin{eqnarray*}
&&\int_{M} f \left((\Delta u + knu)^{2}-|\nabla^{2}u+kug|^{2}\right) dV_g=(n-1)k\int_{M}(\Delta f +nkf)u^{2}dV_g \nonumber \\ &&+\int_{M}\left(\nabla ^{2}f-(\Delta f)g-2(n-1)kfg+fRic\right)(\nabla u, \nabla u)dV_g \nonumber \\
&&+\int_{\partial M}f \left[2\left(\frac{\partial u}{\partial \nu}\right)\Delta_{_{\partial M}}u+H\left(\frac{\partial u}{\partial \nu}\right)^{2}+\mathbb{II}(\nabla_{_{\partial M}}u, \nabla_{_{\partial M}} u)+2(n-1)k\left(\frac{\partial u}{\partial \nu}\right)u\right]dS_g\\
&&+ \int_{\partial M}\frac{\partial f}{\partial \nu}\left(|\nabla _{_{\partial M}}u|^{2}-(n-1)ku^{2}\right) dS_g, \nonumber
\end{eqnarray*} where $\mathbb{II}$ and $H=tr(\mathbb{II})$ stand for the second fundamental form and the mean curvature of $\partial M,$ respectively.
 \end{proposition}

Notice that the classical Reilly's formula is obtained by considering $f=1$ and $k=0$ in the above expression. We refer to \cite[Proposition 1]{Kwong-Miao} for an alternative proof of Proposition \ref{prop-qui-xia}. It is known that the classical Reilly’s formula is particularly efficient for manifolds with nonnegative Ricci curvature. Interestingly, the formula obtained in Proposition \ref{prop-qui-xia} is also useful for manifold that allow negative curvature. In \cite{Kwong-Miao}, Kwong and Miao have used Proposition \ref{prop-qui-xia} to prove a functional inequality on the boundary of static manifolds. Similarly, Ara\'ujo, Freitas and Santos \cite{MFM} used such a proposition to establish an integral inequality for the boundary of a bounded domain in a quasi-Einstein manifold. Di\'ogenes, Pinheiro and Ribeiro \cite{DPR} instead employed the generalized Reilly's formula to prove new sharp integral estimates for critical metrics of the volume functional on compact manifolds with boundary.

Now, we are going to use Proposition \ref{prop-qui-xia} to establish a key lemma that will be applied in the proof of Theorem \ref{theo4}.

\begin{lemma}\label{lema1}
	 Let $(M^n,\,g,\,u,\,\lambda),$ $n\geq 3,$ be a compact oriented $m$-quasi-Einstein manifold with connected boundary and constant scalar curvature. Then we have:
	\begin{equation*}
		\frac{m+1}{m}\int_{M}uRic(\nabla u,\nabla u)dV_g=-\int_{M}u|\mathring{\nabla}^2u|^2dV_g+\frac{(n-2)(n\lambda-R)+n\lambda}{mn}\int_{M}u|\nabla u|^2dV_g.
	\end{equation*}

\end{lemma}
\begin{proof}
We start by using (\ref{eq-laplac}) to infer that $u$ must satisfy

\begin{equation*}
\label{}
\left\{%
\begin{array}{ll}
    \displaystyle \Delta u+n\beta u=0 & \hbox{in $M,$} \\
        \displaystyle u=0 & \hbox{on $\partial M,$} \\
    \end{array}%
\right.
\end{equation*} where $\beta=\frac{(n\lambda-R)}{mn}.$ Hence, choosing $u=f$ and $k=\beta,$ it follows from Proposition \ref{prop-qui-xia} that
\begin{eqnarray*}
	-\int_{M}u\left| \nabla^2u-\frac{\Delta u}{n}g\right|^2 dV_g=\int_{M}\left(\nabla^2u-(\Delta u)g+\frac{2(n-1)}{n}(\Delta u)g+uRic\right)(\nabla u,\nabla u)dV_g,
	\end{eqnarray*} where we also have used that $u=0$ on $\partial M.$ Therefore, by (\ref{eq1F2}), one obtains that
	\begin{eqnarray*}
	-\int_{M}u|\mathring{\nabla}^2u|^2dV_g&=&\int_{M}\left (\frac{u}{m}(Ric-\lambda g)+\frac{(n-2)}{n}(\Delta u)g+uRic\right)(\nabla u,\nabla u)dV_g\\
	&=&\frac{m+1}{m}\int_{M}uRic(\nabla u,\nabla u)dV_g-\frac{\lambda}{m}\int_{M}u|\nabla u|^2dV_g\\
	&&+\frac{(n-2)}{n}\int_{M}(\Delta u)|\nabla u|^2dV_g\\
	&=&\frac{m+1}{m}\int_{M}uRic(\nabla u,\nabla u)dV_g-\frac{\lambda}{m}\int_{M}u|\nabla u|^2dV_g\\
	&&+\frac{(n-2)(R-n\lambda)}{mn}\int_{M}u|\nabla u|^2dV_g,
\end{eqnarray*} where in the last equality we have used again (\ref{eq-laplac}). So, the proof is completed. 

\end{proof}

To conclude this section, we recall another {\it generalized Reilly's formula} obtained subsequently by Li and Xia in \cite[Theorem 1.1]{LX}.

\begin{proposition}[\cite{LX}]
\label{propJC} 
Let $(M^n,\,g)$ be an $n$-dimensional smooth Riemannian manifold and $\Omega\subset M$ be a bounded domain with smooth boundary $\partial \Omega.$ Let $V\in$ $C^{\infty}(\overline{\Omega})$ be a smooth function such that $\frac{\nabla^2V}{V}$ is continuous up to $\partial \Omega$. Then for any $f\in C^{\infty}(\overline{\Omega})$, the following integral identity holds:
	\begin{eqnarray}
	\label{lkj19}
		&&\int_{\Omega}V\left(\left( \Delta f-\frac{\Delta V}{V}f\right)^2-\left |\nabla^2f-\frac{\nabla^2V}{V}f\right|^2\right)d\Omega\nonumber\\
		&=& \int_{\partial \Omega}\left(V\mathbb{II}(\nabla_{_{\partial \Omega}} f,\nabla_{_{\partial \Omega}} f)+2V\frac{\partial f}{\partial \nu}\Delta_{_{\partial \Omega}} f+VH\left (\frac{\partial f}{\partial \nu}\right)^2+\frac{\partial V}{\partial \nu}|\nabla_{_{\partial \Omega}} f|^2+2f\nabla^2V(\nabla_{_{\partial \Omega}} f,\nu)\right) dS_{g}\nonumber\\
		&&+\int_{\partial \Omega}\left(-2f\frac{\partial f}{\partial \nu}\left(\Delta_{_{\partial \Omega}} V+H\frac{\partial V}{\partial \nu}\right)-f^2\frac{\nabla^2V-\Delta Vg}{V}(\nabla V,\nu)\right)dS_{g}\nonumber\\
		&&+\int_{\Omega}(\Delta V g-\nabla^2V+VRic)\left (\nabla f-\frac{\nabla V}{V}f,\nabla f-\frac{\nabla V}{V}f\right)d\Omega.\nonumber
	\end{eqnarray}
	Here, $\nu$ is the outward unit normal vector, $\mathbb{II}$ e $H=tr(\mathbb{II})$ are the second fundamental form and the mean curvature of $\partial \Omega,$ respectively. 
\end{proposition}

Notice that the classical Reilly’s formula is obtained by considering $V = 1.$

\section{Proof of the main results}
\label{sec3}

In this section, we shall present the proofs of Theorems \ref{Theo2}, \ref{theo4}, \ref{thmJac} and \ref{teoHK1} and Corollaries \ref{corol1} and \ref{corJac}.

\subsection{Proof of Theorem \ref{Theo2}}
\begin{proof} To begin with, upon integrating (\ref{eqmu}) over $M,$ we use (\ref{eq-laplac}) to infer
\begin{align}\label{eq3.2}
	\mu Vol(M)=&\int_{M}\frac{u^2}{m}(R-n\lambda)dV_g+(m-1)\int_{M}|\nabla u|^2dV_g+\lambda\int_{M}u^2dV_g\nonumber \\
	\geq&-\frac{n\lambda}{m+n-1}\int_{M}u^2dV_g+(m-1)\int_{M}|\nabla u|^2dV_g+\lambda\int_{M}u^2dV_g\nonumber\\
	=&\frac{m-1}{m+n-1}\lambda\int_{M}u^2dV_g+(m-1)\int_{M}|\nabla u|^2dV_g,
\end{align}
where in the second line we have used the estimate (\ref{limitacaoR}). Next, from the Rayleigh-quotient characterization of the first nonzero eigenvalue of the
Laplacian $\lambda_1,$ one sees that

\begin{align*}
	\int_M|\nabla u|^2dV_g\geq \lambda_1\int_Mu^2dV_g,
\end{align*} where $$\lambda_1= \underset{H_0^{1,2}(M), u\not\equiv 0}{\inf} \frac{\int_{M}|\nabla u|^2dV_g}{\int_{M}u^2dV_{g}}.$$
Plugging this fact into (\ref{eq3.2}) yields
\begin{equation*}
	\mu Vol(M)\geq\frac{m-1}{m+n-1} \left(\lambda_1(m+n-1)+\lambda \right)\int_M u^2 dV_g.
\end{equation*} By Holder's inequality, one obtains that
\begin{equation}\label{eq3.4}
	\mu Vol(M)^2\geq \frac{(m-1)}{(m+n-1)}\left(\lambda_{1}(m+n-1)+\lambda\right)\left (\int_{M}udV_g\right)^2.
\end{equation}

On the other hand, combining (\ref{eq-laplac}) and (\ref{limitacaoR}), we deduce

\begin{equation}
\Delta u\geq -\frac{n\lambda }{m+n-1}u,
\end{equation} and by Stokes' theorem, one sees that $$\int_{M}u\, dV_{g}\geq \frac{m+n-1}{n\lambda}\int_{M}\left(-\Delta u\right) dV_{g}=\frac{m+n-1}{n\lambda}|\nabla u|_{\mid_{_{\partial M}}}|\partial M|.$$ Substituting this into (\ref{eq3.4}), we arrive at

\begin{align*}
	\mu Vol(M)^2\geq&
	(m-1)\frac{\left(\lambda_{1}(m+n-1)+\lambda\right)(m+n-1)}{n^2 \lambda^2}|\nabla u|^{2}_{\mid_{_{\partial M}}} |\partial M|^2.
\end{align*} Next, since $u|_{\partial M}=0,$ by evaluating (\ref{eqmu}) on $\partial M,$ one sees that $\mu=(m-1)|\nabla u|^{2}_{\mid_{_{\partial M}}}.$ Therefore, rearranging terms, we have

\begin{align}
\label{eq3.511}
	|\partial M|\leq \frac{n\lambda}{\sqrt{\left(\lambda_1(m+n-1)+\lambda\right)(m+n-1)}}Vol(M),
	\end{align} which proves (\ref{eqTeo2}). Moreover, if equality holds in (\ref{eq3.511}), then 
\begin{align*}
	R=\frac{n(n-1)}{m+n-1}\lambda.
\end{align*} Hence, it follows from (\ref{laplacianoR}) that $|\mathring{Ric}|=0$ on $M$, i.e., $M^n$ is Einstein and in this case, it suffices to apply Proposition 3.1 of \cite{He-Petersen-Wylie} to conclude that $M^n$ is isometric to the standard hemisphere $\mathbb{S}^n_+.$ This finishes the proof of the theorem. 
\end{proof}

\subsection{Proof of Theorem \ref{theo4}}

\begin{proof}
	Initially, by the classical B\"ochner's formula, we have
	\begin{align*}
		u\Delta|\nabla u|^2=2uRic(\nabla u,\nabla u)+2u\langle \nabla(\Delta u),\nabla u\rangle+2u|\nabla^2u|^2.
	\end{align*} Upon integrating this expression over $M,$ we use (\ref{eq-laplac}) and the fact that $M$ has constant scalar curvature in order to infer
	\begin{align}\label{eq3.5}
		\int_{M}u\Delta|\nabla u|^2dV_g=&2\int_{M}uRic(\nabla u,\nabla u)dV_g-\frac{2\alpha}{m}\int_{M}u|\nabla u|^2dV_g+2\int_{M}u|\nabla^2u|^2dV_g\nonumber \\
		=&2\int_{M}uRic(\nabla u,\nabla u)dV_g-\frac{2\alpha}{m}\int_{M}u| \nabla u|^2dV_g+2\int_{M}u|\mathring{\nabla}^2u|^2dV_g\nonumber\\&+\frac{2}{n}\int_{M}u(\Delta u)^2dV_g,
		\end{align} where $\mathring{\nabla}^2 u=\nabla^2 u -\frac{\Delta u}{n}g$ and $\alpha=n\lambda-R$.

		On the other hand, by using the Stokes' theorem and the fact that $u=0$ on $\partial M,$ one sees that
	\begin{align*}
	\int_{M}u\Delta|\nabla u|^2 \,dV_g=&\int_{M}|\nabla u|^2\Delta u \,dV_g\nonumber \\
	&+\int_{\partial M}\left(u\left \langle \nabla |\nabla u|^2,-\frac{\nabla u}{|\nabla u|}\right\rangle-|\nabla u|^2\left\langle \nabla u,-\frac{\nabla u}{|\nabla u|}\right\rangle\right)\,dS_g\nonumber \\
	=&-\frac{\alpha}{m}\int_{M}u|\nabla u|^2\,dV_g+|\nabla u|^3_{\mid_{_{\partial M}}}|\partial M|,
\end{align*} which compared with (\ref{eq3.5}) and using once more (\ref{eq-laplac}) yields
\begin{align*}
	|\nabla u|^3_{\mid_{_{\partial M}}}|\partial M|=&2\int_{M}uRic(\nabla u,\nabla u)dV_g-\frac{\alpha}{m}\int_{M}u| \nabla u|^2dV_g\\&+2\int_{M}u|\mathring{\nabla}^2u|^2dV_g\nonumber+\frac{2}{n}\int_{M}u(\Delta u)^2dV_g\\
	=&2\int_{M}uRic(\nabla u,\nabla u)dV_g-\frac{\alpha}{m}\int_{M}u| \nabla u|^2dV_g\\&+2\int_{M}u|\mathring{\nabla}^2u|^2dV_g\nonumber+\frac{2\alpha^2}{nm^2}\int_{M}u^3dV_g.
\end{align*}
	Now, we may invoke Lemma \ref{lema1} to infer
	\begin{align}\label{eq3.6}
	|\nabla u|^3_{\mid_{_{\partial M}}} |\partial M|=&\frac{2}{m+1}\int_{M}u|\mathring{\nabla}^2u|^2dV_g+\frac{2\alpha^2}{m^2n}\int_Mu^3dV_g\nonumber\\
	&+\frac{2mn\lambda +\alpha(n(m-1)-4m)}{n(m+1)m}\int_{M}u|\nabla u|^2dV_g.
	\end{align}
	
At the same time, by (\ref{eq-laplac}), Stokes' theorem and the fact that $u|_{\partial M}=0,$ we have
\begin{align*}
\frac{\alpha}{m}\int_{M}u^3dV_g=	-\int_{M}u^2(\Delta u)\,dV_{g}=\int_{M}\langle \nabla u^2,\nabla u\rangle\, dV_{g}=2\int_{M}u|\nabla u|^2\,dV_{g}.
\end{align*} Substituting this into (\ref{eq3.6}) gives

	\begin{align*}
	|\nabla u|^3_{\mid_{_{\partial M}}} |\partial M| =&\frac{2}{m+1}\int_{M}u|\mathring{\nabla}^2u|^2dV_g+\frac{\left(2nm\lambda+\alpha(n(m-1)+4)\right)\alpha}{2nm^2(m+1)}\int_{M}u^3dV_g,
	\end{align*}
so that,
	\begin{align}
	\label{eq3.7}
			|\nabla u|^3_{\mid_{_{\partial M}}} |\partial M| \geq \frac{\left(2nm\lambda+\alpha(n(m-1)+4)\right)\alpha}{2nm^2(m+1)}\int_{M}u^3dV_g.
	\end{align}

Proceeding, upon integrating (\ref{eqmu}) over $M^n,$ we use once more the Stokes' theorem and  (\ref{eq-laplac}) to infer 
\begin{eqnarray}\label{eq3.8}
	\mu Vol(M)&=&\int_{M}u\Delta u\,dV_{g}+(m-1)\int_{M}|\nabla u|^2 dV_{g}+\lambda \int_{M}u^2 \, dV_{g}\nonumber\\&=&-(m-2) \int_{M}u\Delta u\,dV_{g}+\lambda\int_{M}u^2\,dV_{g}\nonumber\\
	&=&\left(\frac{(m-2)\alpha+\lambda m}{m}\right)\int_{M}u^2\,dV_{g}.
\end{eqnarray} Besides, by Holder's inequality and (\ref{eq-laplac}), one sees that
\begin{align}\label{eq3.9}
\int_{M}u^2dV_g=&\int_{M}u^{\frac{3}{2}}u^{\frac{1}{2}}dV_g\nonumber\\
\leq&\left (\int_{M}u^3dV_g\right)^{\frac{1}{2}}\left(\int_{M}udV_g\right)^{\frac{1}{2}}\nonumber\\
=&\left (\int_{M}u^3dV_g\right)^{\frac{1}{2}}\left(\frac{m}{\alpha}\int_{M}(-\Delta u)dV_g\right)^{\frac{1}{2}}\nonumber\\
=&\left(\int_{M}u^3dV_g\right)^{\frac{1}{2}}\left(\frac{m|\nabla u|_{\mid_{_{\partial M}}}}{\alpha}|\partial M|\right)^{\frac{1}{2}},
\end{align}
where we have used that $\alpha=n\lambda-R$ is positive, which follows from the fact that $M^n$ has constant scalar curvature (see Remark \ref{rem2a}, or \cite[Corollary 4.3]{He-Petersen-Wylie}). Next, we already know from the proof of Theorem \ref{Theo2} that $\mu=(m-1)|\nabla u|^{2}_{\mid_{_{\partial M}}}$ and hence, plugging (\ref{eq3.9}) into (\ref{eq3.8}), one concludes that
\begin{align}
		Vol(M)\leq\left(\frac{(m-2)\alpha+\lambda m}{(m-1)\sqrt{m\alpha}|\nabla u|^{\frac{3}{2}}_{\mid_{_{\partial M}}}}\right) |\partial M|^{\frac{1}{2}}\left(\int_{M}u^3dV_g\right)^{\frac{1}{2}}.
\end{align}
 Whence, it suffices to invoke (\ref{eq3.7}) to achieve
\begin{equation}
\label{pljk13}
	Vol(M)\leq \left(\frac{(m-2)\alpha+\lambda m}{(m-1)\alpha}\right)\sqrt{\frac{2n(m+1)m}{2\lambda mn+\alpha[n(m-1)+4]}}\,|\partial M|,
\end{equation} which proves the stated inequality.

Finally, if equality holds in (\ref{pljk13}), then (\ref{eq3.7}) also becomes an equality. It follows that $|\mathring{\nabla}^2u|^2=0$ and by using (\ref{RicciSemtrac}), one sees that $M^n$ is an Einstein manifold. Thereby, we are in position to apply Proposition 3.1 of \cite{He-Petersen-Wylie} to conclude that $M^n$ is isometric, up to scaling, to the standard hemisphere $\mathbb{S}^n_+.$ So, the proof is completed. 
\end{proof}

\subsection{Proof of Corollary \ref{corol1}}
	\begin{proof} To begin with, one observes that
		\begin{align*}
			|\partial M|^n=&|\partial M||\partial M|^{n-1}.\\
			=&\frac{\alpha}{m|\nabla u|_{\mid_{_{\partial M}}}}\left(\int_{M}udV_g\right) |\partial M|^{n-1},
		\end{align*}
	where we have used (\ref{eq-laplac}) and the Stokes' theorem.
		Consequently, by Theorem \ref{theo4}, one sees that
	\begin{equation}
		|\partial M|\geq \Lambda(\alpha,m,n,u)^{\frac{1}{n}}Vol(M)^{\frac{n-1}{n}},
	\end{equation} where 
\begin{equation*}
\label{}
	\Lambda(\alpha,m,n,u) =\frac{\alpha}{m|\nabla u|_{\mid_{_{\partial M}}}}\left(\frac{(m-1)\alpha}{(m-2)\alpha+\lambda m}\sqrt{\frac{2\lambda mn+(\alpha(n(m-1)+4)}{2n(m+1)m}}\right)^{n-1}\int_{M}udV_g,
\end{equation*} as asserted. Furthermore, the case of equality follows directly from Theorem \ref{theo4}. Thus, we finish the proof of the corollary. 
	\end{proof}
	
	\subsection{Proof of Theorem \ref{thmJac}}
	
	\begin{proof}
	We start by claiming that $R^{\partial M}>0.$ Indeed, it follows from Lemma \ref{lem1A1} and (\ref{limitacaoR}) that
	
	\begin{eqnarray*}
	(m-1)R^{\partial M}&=& (m+1)R-2(n-1)\lambda\nonumber\\&\geq & \frac{(m+1)n(n-1)}{m+n-1}\lambda -2(n-1)\lambda\nonumber\\&=& \frac{(m-1)(n-1)(n-2)}{m+n-1}\lambda.
	\end{eqnarray*} Since $m>1$ and $M^n$ is compact, we then obtain 
	\begin{equation}
	\label{eqKpo89}
	R^{\partial M}\geq \frac{(n-1)(n-2)}{m+n-1}\lambda >0,
	\end{equation} as claimed. 
	
	From now on, we adapt the arguments by Barros-Silva \cite{BS}. Firstly, since $\partial M$ is totally geodesic, we obtain from (\ref{eq-egvle}) that 
	
	\begin{eqnarray*}
	\beta_{1}\int_{\partial M}\varphi^2 dS_{g}\leq -\int_{\partial M}\varphi J(\varphi) dS_{g}=\int_{\partial M} |\nabla^{\partial M}\varphi|^2 dS_{g}-\int_{\partial M}R_{nn}\varphi^2 dS_{g}
		\end{eqnarray*} for any $\varphi\in C^{\infty}(\partial M).$ Thereby, choosing $\varphi\equiv 1,$ one sees that
		
		\begin{eqnarray}
		\label{eqER1a}
		\beta_{1}|\partial M| &\leq & -\int_{\partial M}R_{nn}dS_{g}=-\frac{1}{2}\int_{\partial M}\left(R-R^{\partial M}\right) dS_{g}\nonumber\\ &\leq & \frac{1}{2}\int_{\partial M}R^{\partial M} dS_{g} -\frac{n(n-1)}{2(m+n-1)}\lambda |\partial M|,
		\end{eqnarray} where we have used (\ref{EscalarBordoM}) and (\ref{limitacaoR}).

		On the other hand, since $\partial M$ is Einstein and $R^{\partial M}>0,$ we may write $$Ric^{\partial M}=\frac{R^{\partial M}}{n-1}=(n-2)\varepsilon,$$ where $\varepsilon=\frac{R^{\partial M}}{(n-1)(n-2)}.$ Whence, it follows from Bonnet-Myer's theorem that $diam_{g_{\partial M}}(\partial M)\leq \frac{\pi}{\sqrt{\varepsilon}}.$ Furthermore, by Bishop-Gromov's theorem, we have
		
		\begin{equation}
		\label{eqEr123}
		Vol\left(B_{\frac{\pi}{\sqrt{\varepsilon}}}^{\partial M}\right)(p)\leq Vol\left(\Bbb{S}_{g_{\varepsilon}}^{n-1}\right)=\varepsilon^{-\frac{(n-1)}{2}}\omega_{n-1},
		\end{equation} for any point $p\in \partial M,$ where $g_{\varepsilon}=\frac{1}{\varepsilon}g_{\Bbb{S}^{n-1}}$ and $\omega_{n-1}$ denotes the volume of the standard unit sphere $\Bbb{S}^{n-1}.$ These facts together yield
		
		\begin{eqnarray*}
		|\partial M|\leq Vol\left(B_{\frac{\pi}{\sqrt{\varepsilon}}}^{\partial M}\right)\leq \varepsilon^{-\frac{(n-1)}{2}}\omega_{n-1},
		\end{eqnarray*} so that 
		
		\begin{equation}
		\label{eqK130p}
		R^{\partial M}\leq (n-1)(n-2)\left(\omega_{n-1}\right)^{\frac{2}{n-1}}|\partial M|^{-\frac{2}{n-1}}.
		\end{equation} This substituted into (\ref{eqER1a}) gives
		
		\begin{eqnarray*}
		2\beta_{1}\leq (n-1)(n-2)\left(\omega_{n-1}\right)^{\frac{2}{n-1}}|\partial M|^{-\frac{2}{n-1}}-\frac{n(n-1)}{m+n-1}\lambda.
		\end{eqnarray*} Consequently, 
		
		\begin{equation}
		\label{eqth6p}
		\left(\frac{2}{n-1}\beta_{1}+\frac{n}{m+n-1}\lambda\right)|\partial M|^{\frac{2}{n-1}}\leq (n-2)\left(\omega_{n-1}\right)^{\frac{2}{n-1}}.
		\end{equation} Now, it suffices to use Remark \ref{remZ} to conclude that $\frac{2}{n-1}\beta_{1}+\frac{n}{m+n-1}\lambda>0$ and hence, the stated inequality follows.

	Finally, if equality holds in (\ref{eqth6p}), then (\ref{eqEr123}) also becomes an equality. Thus, it follows from the equality case for the Bishop-Gromov’s theorem that the boundary is isometric to a round sphere $\Bbb{S}^{n-1}.$ This finishes the proof of the theorem. 
	
	\end{proof}

	\subsection{Proof of Corollary \ref{corJac}}
	\begin{proof}
	Firstly, we invoke the Gauss-Bonnet formula to deduce
	
	\begin{eqnarray*}
	\frac{1}{2}\int_{\partial M}R^{\partial M} dS_{g}=\int_{\partial M}K^{\partial M} dS_{g}=2\pi \chi(\partial M).
	\end{eqnarray*} So, it suffices to use (\ref{eqKpo89}) to conclude that $\chi(\partial M)>0$ and $\partial M$ is a $2$-sphere. Next, since $\partial M$ is Einstein, we may combine (\ref{eqK130p}) and (\ref{eqKpo89}) in order to infer 
	
	\begin{equation}
	\label{eq89n10}
	|\partial M|\leq \frac{m+2}{\lambda}4\pi.
	\end{equation}

	On the other hand, from (\ref{limitacaoR}) we have 
	
	\begin{equation}
	\label{eqpl14}
	\Lambda\geq \frac{6}{m+2}\lambda.
	\end{equation} Together with the facts that $\chi(\partial M)=2$ and that $\partial M$ is totally geodesic, this implies that 
	\begin{eqnarray}
	\label{eekjl1}
	\mathfrak{m}_{H}(\partial M)&=&\frac{|\partial M|^{\frac{1}{2}}}{(16\pi)^{\frac{3}{2}}}\left(16\pi-\frac{2}{3}\Lambda |\partial M|\right)\nonumber\\
	&\leq & \frac{|\partial M|^{\frac{1}{2}}}{(16\pi)^{\frac{1}{2}}} -\frac{4}{m+2}\lambda \left(\frac{|\partial M|}{16\pi}\right)^{\frac{3}{2}}\nonumber\\&\leq & \frac{1}{2}\sqrt{\frac{(m+2)}{\lambda}}-\frac{4}{m+2}\lambda \left(\frac{|\partial M|}{16\pi}\right)^{\frac{3}{2}},
	\end{eqnarray} where in the last inequality we have used  (\ref{eq89n10}). This therefore gives the first stated inequality. Moreover, if equality holds in (\ref{eekjl1}), it suffices to invoke the equality case of Theorem \ref{thmJac}.

	We now deal with the second inequality. It follows from the Hawking mass definition and the fact that $\chi(\partial M)=2$ that 
	
	\begin{equation}
	\label{eq8901ap}
	\mathfrak{m}_{H}(\partial M)=\frac{4|\partial M|^{\frac{1}{2}}}{\left(16\pi\right)^{\frac{3}{2}}}\left(4\pi -\frac{\Lambda}{6}|\partial M|\right).
	\end{equation}
	At the same time, choosing $\varphi\equiv 1$ in (\ref{eq-egvle}), we then use (\ref{EscalarBordoM}) and the Gauss-Bonnet formula in order to deduce
	
	\begin{eqnarray}
	\beta_{1}|\partial M|&\leq& -\int_{\partial M}R_{nn} dS_{g}\nonumber\\&=&\int_{\partial M}K^{\partial M} dS_{g}-\frac{1}{2}\int_{\partial M}R dS_{g}\nonumber\\&=& 4\pi -\frac{1}{2}\int_{\partial M}R dS_{g}\nonumber\\&\leq &4\pi -\frac{\Lambda}{2}|\partial M|\nonumber.
	\end{eqnarray} This substituted into (\ref{eq8901ap}) yields
	
	$$\mathfrak{m}_H(\partial M)\geq 4\left(\beta_{1}+\frac{\Lambda}{3}\right) \left(\frac{|\partial M|}{16\pi}\right)^{\frac{3}{2}}.$$ Now, we use (\ref{eqpl14}) to infer 	
	
	$$\mathfrak{m}_H(\partial M)\geq 4\left(\beta_{1}+\frac{2\lambda}{m+2}\right) \left(\frac{|\partial M|}{16\pi}\right)^{\frac{3}{2}} \geq 0,$$ where we have used Proposition \ref{proklk} to conclude that $\beta_{1}+\frac{2\lambda}{m+2}\geq 0.$ Thus, the proof is finished.

	\end{proof}

	\subsection{Proof of Theorem \ref{teoHK1}}
	\begin{proof}
	Initially, we consider the auxiliar tensor $$P=Ric-\frac{(n-1)\lambda -R}{m-1}g.$$ With aid of this notation, we see from (\ref{fund-equation}) and (\ref{eq-laplac}) that

\begin{eqnarray*}
(\Delta u)g-\nabla^2 u+uRic&=& \frac{u}{m}\left[\left(R-(n-1)\lambda\right)g +(m-1)Ric\right]\nonumber\\&=&\frac{(m-1)}{m}u P. 
\end{eqnarray*} In particular, since $m>1,$ $M^n$ is a substatic manifold if and only if $P\geq 0.$ 

On the other hand, it is not hard to check that the Ricci bound assumption implies

\begin{eqnarray*}
P&\geq & \left(\frac{(n-1)\lambda}{m+n-1}-\frac{\left( (n-1)\lambda -R\right)}{m-1}\right)g \nonumber\\&=&\frac{R(m+n-1)-n(n-1)\lambda}{(m-1)(m+n-1)}g.
\end{eqnarray*} By tracing the assumption on Ricci, one sees that

\begin{equation}
\label{eq3.15}
R(m+n-1)\geq n(n-1)\lambda.
\end{equation} Consequently, $P\geq 0 $ and hence, $M^n$ is a substatic Riemannian manifold. Thus, the Heintze-Karcher type inequality follows from Theorem 1.3 in \cite{LX}.

At the same time, we are going to present here another proof of the asserted inequality in order to discuss the equality case. To do so, we divide the proof into two parts and adapt the arguments by Li-Xia \cite{LX}. We start by considering the following boundary value problem

	\begin{equation}
\label{eq5}
\left\{%
\begin{array}{ll}
    \displaystyle \Delta f-\frac{\Delta u}{u}f=1 & \hbox{in $\Omega,$} \\
        \displaystyle f=0 & \hbox{on $\partial \Omega.$} \\
    \end{array}%
\right.
\end{equation} 
Notice that by \cite[Lemma 2.5]{LX}, the first eingenvalue $\lambda_1(\Delta -q,\Omega)>0,$ where $q:=\frac{\Delta u}{u}.$ Consequently, by the standard elliptic PDE theory (see \cite{BNV}), (\ref{eq5}) admits a unique smooth solution $f\in$ $C^{\infty}(\overline{\Omega}).$ Thereby, using Proposition \ref{propJC} for such a solution and choosing $V=u,$ we see that

\begin{align}
\label{eq6}
	\int_{\Omega}ud\Omega=&\int_{\Omega}u\left |\nabla^2f-\frac{\nabla^2u}{u}f\right|^2d\Omega+\int_{\partial \Omega}uH\left (\frac{\partial f}{\partial \nu}\right)^2 dS_{g}\nonumber\\
	&+\int_{\Omega}(\Delta u g-\nabla^2u+uRic)\left(\nabla f-\frac{\nabla u}{u}f,\nabla f-\frac{\nabla u}{u}f\right)d\Omega.
\end{align} By using the fact that 

\begin{equation}
\label{eqYU92}
	\left|\nabla^2f-\frac{\nabla^2u}{u}f\right|^2\geq \frac{1}{n}\left|\Delta f-\frac{\Delta u}{u}f\right|^2=\frac{1}{n},
\end{equation} we then obtain

\begin{align}
\label{eq3.14}
	\frac{n-1}{n}\int_{\Omega}ud\Omega\geq& \int_{\partial \Omega}uH\left (\frac{\partial f}{\partial \nu}\right)^2 dS_{g}\nonumber\\
	&+\int_{\Omega}(\Delta u g-\nabla^2u+uRic)\left(\nabla f-\frac{\nabla u}{u}f,\nabla f-\frac{\nabla u}{u}f\right)d\Omega.
\end{align}

Since $P\geq 0,$ one obtains that 
\begin{eqnarray}
\label{eq3.16}
	\frac{n-1}{n}\int_{\Omega}ud\Omega
	&\geq&\int_{\partial \Omega}uH\left (\frac{\partial f}{\partial \nu}\right)^2 dS_{g}.
\end{eqnarray}

Next, it follows from (\ref{eq5}) that $u\Delta f-(\Delta u)f=u$ in $\Omega.$ Thus, by Green's identity, one obtains that
\begin{eqnarray*}
	\int_{\Omega}ud\Omega&=&\int_{\Omega}u\Delta fd\Omega-\int_{\Omega}(\Delta u)fd\Omega\\
	&=&\int_{\partial \Omega}u\frac{\partial f}{\partial \nu}dS_{g}-\int_{\partial \Omega}f\frac{\partial u}{\partial \nu}dS_{g}\\
	&=&\int_{\partial \Omega}u\frac{\partial f}{\partial \nu}dS_{g},
\end{eqnarray*} where in the last line we used that  $f|_{\partial \Omega}=0.$ By H\"older's inequality and (\ref{eq3.16}), we have
\begin{eqnarray}
	\left(\int_{\Omega}ud\Omega \right)^2&=&\left(\int_{\partial \Omega}u\frac{\partial f}{\partial \nu}dS_{g} \right)^2\leq \int_{\partial \Omega}uH\left (\frac{\partial f}{\partial \nu}\right)^2 dS_{g}\int_{\partial \Omega}\frac{u}{H} dS_{g}\nonumber \\
	&\leq& \frac{n-1}{n}\int_{\Omega}ud\Omega \int_{\partial \Omega}\frac{u}{H} dS_{g}\nonumber,
\end{eqnarray} so that

\begin{align*}
	n\int_{\Omega}ud\Omega\leq (n-1)\int_{\partial \Omega}\frac{u}{H}dS_{g},
\end{align*} which proves the asserted inequality.

Now, we deal with the equality case. Indeed, if equality holds, one deduces from (\ref{eqYU92}) and (\ref{eq3.15}) that
\begin{align}
\label{eq3.17}
\nabla^2f-\frac{\nabla^2u}{u}f=\frac{1}{n}g
\end{align}
and 
\begin{align*}
		R=\frac{n(n-1)}{m+n-1}\lambda
	\end{align*} in $\Omega.$ Since $u$ and $g$ are real analytic in harmonic coordinates (cf. Proposition 2.4 in \cite{He-Petersen-Wylie}), we conclude that $$R=\frac{n(n-1)}{m+n-1}\lambda$$ in $M.$ This into (\ref{laplacianoR}) therefore implies that $(M^n,\,g)$ is an Einstein manifold. 
	
	We now have two cases to be analyzed, namely, $\lambda=0$ and $\lambda\neq 0.$ Firstly, if $\lambda=0,$ it suffices to invoke Proposition 3.1 of \cite{He-Petersen-Wylie} to conclude that $M^n$ is isometric to $[0,\infty)\times F,$ where $F$ is Ricci flat. In particular, we have $\nabla^2 u=\frac{u}{m}(Ric-\lambda g)=0$ in $M,$ which compared with (\ref{eq3.17}), gives $\nabla^2f=\frac{1}{n}g$ in $\Omega.$ Taking into account that $f|_{\partial \Omega}=0,$ it suffices to apply Lemma 3 of \cite{Reilly} to conclude that $\Omega$ is a geodesic ball.
	
On the other hand, if $\lambda\neq 0,$ then we may use again Proposition 3.1 of \cite{He-Petersen-Wylie} to infer that $M^n$ is isometric to either $\mathbb{S}_+^n,$ or $\mathbb{H}^n,$ or $[0,\infty)\times N,$ or $\mathbb{R}\times F.$ In particular, one sees that 
\begin{equation*}
	\frac{\nabla^2u}{u}=\frac{1}{m}(Ric_g-\lambda g)=-\frac{\lambda}{m+n-1}g
\end{equation*} and by (\ref{eq3.17}), we have
\begin{equation*}
	\nabla^2f=-\frac{\lambda}{m+n-1}fg+\frac{1}{n}g.
\end{equation*} By considering $F:=f-K$, where $K=\frac{m+n-1}{n\lambda},$ it is easy to check that

	\begin{equation*}
	\nabla^2F=-\frac{\lambda}{m+n-1}Fg
	\end{equation*} in $\Omega.$ Moreover, $F|_{\partial \Omega}$ is constant. Therefore, it suffices to apply Theorem B of \cite{Reilly} to conclude that $\Omega$ is a geodesic ball of constant sectional curvature  $\frac{\lambda}{m+n-1}.$ So, the proof is completed. 
	\end{proof}

\begin{acknowledgement}
The authors would like to thank the referee for the careful reading and valuable suggestions. They are also grateful to B. Leandro for his interest in this work and for pointing out the preprint \cite{Benedito}. In addition, they thank A. Barros and N. Pinheiro for the helpful conversations on the subject.
\end{acknowledgement}

\end{document}